\documentclass[12pt]{amsart}
\pdfoutput=1
\newif\iffinal
\finaltrue
\iffinal
\else
\usepackage[eso-foot]{svninfo}
\svnInfo $Id$
\fi

\usepackage{amsmath,amsfonts,amsthm,amssymb}
\usepackage{geometry}
\usepackage{stmaryrd}
\usepackage{hyperref}
\usepackage{tikz}
\usepackage{lineno}
\usepackage{pgfplots}
\pgfplotsset{compat=1.5}

\newcommand{\N}{\mathbb{N}}

\newcommand{\Prob}{\mathbb{P}}
\newcommand{\Z}{\mathbb{Z}}
\newcommand{\C}{\mathbb{C}}
\newcommand{\E}{\mathbb{E}}
\newcommand{\V}{\mathbb{V}}
\newcommand{\DLMF}[2]{\cite[\href{http://dlmf.nist.gov/#1.E#2}{#1.#2}]{NIST:DLMF}}

\DeclareMathOperator{\Res}{Res}

\theoremstyle{plain}
\newtheorem{theorem}{Theorem}
\newtheorem{corollary}{Corollary}[section]
\newtheorem{lemma}[corollary]{Lemma}
\newtheorem{proposition}[corollary]{Proposition}

\theoremstyle{definition}
\newtheorem{definition}[corollary]{Definition}

\renewcommand{\MR}[1]{}

\theoremstyle{remark}
\newtheorem*{remark}{Remark}

\begin{document}

\title[Ballot Sequences and Random Walks]{Analysis of Bidirectional Ballot Sequences and Random Walks Ending in their Maximum}
\author[B.~Hackl]{Benjamin Hackl}
\author[C.~Heuberger]{Clemens Heuberger}

\address[Benjamin Hackl, Clemens Heuberger]{Institut f\"ur Mathematik,
  Alpen-Adria-Uni\-ver\-si\-t\"at Klagenfurt, Universit\"atsstra\ss e
  65--67, 9020 Klagenfurt, Austria} 
\email{\href{mailto:benjamin.hackl@aau.at}{benjamin.hackl@aau.at}}
\email{\href{mailto:clemens.heuberger@aau.at}{clemens.heuberger@aau.at}}
\thanks{B.~Hackl and C.~Heuberger are supported by the Austrian
  Science Fund (FWF): P~24644-N26.}

\author[H.~Prodinger]{Helmut Prodinger}
\thanks{H.~Prodinger is supported by an incentive grant of the
  National Research Foundation of South Africa.}

\author[S.~Wagner]{Stephan Wagner}
\address[Helmut Prodinger, Stephan Wagner]{Department of Mathematical
  Sciences, Stellenbosch University, 7602 Stellenbosch, 
 South Africa}
\email{\href{mailto:hproding@sun.ac.za}{hproding@sun.ac.za}}
\email{\href{mailto:swagner@sun.ac.za}{swagner@sun.ac.za}}
\thanks{S.~Wagner is supported by the National Research Foundation of
  South Africa, grant number 70560.}

\iffinal\else
\linenumbers 
\fi
\begin{abstract}
Consider non-negative lattice paths ending at their maximum height, which will be called admissible paths. We show that the probability for a lattice path to be admissible is related to the Chebyshev polynomials of the first or second kind,
depending on whether the lattice path is defined with a reflective
barrier or not. Parameters like the number of admissible paths with given length
or the expected height are analyzed asymptotically. Additionally, we
use a bijection between admissible random walks and special binary
sequences to prove a recent conjecture by Zhao on ballot sequences.
\end{abstract}
\maketitle

\section{Introduction}

Lattice paths as well as their stochastic incarnation---random walks---are interesting and classical objects of study. Several authors have investigated a variety of parameters related to lattice paths. For example, Banderier and Flajolet gave an
asymptotic analysis of the number of special lattice paths with fixed
length in \cite{Banderier-Flajolet:2002:lat-path}. De Brujin, Knuth,
and Rice \cite{Bruijn-Knuth-Rice:1972} analyzed the expected height of
certain lattice paths, and Panny and Prodinger
\cite{Panny-Prodinger:1985:path-height} determined the asymptotic behavior
of such paths with respect to several notions of height. 

The particular class of lattice paths we want to analyze in this paper is defined
as follows.

\begin{definition}[Admissible random walks and lattice paths]
  Let $(S_{k})_{0\leq k \leq n}$ be a simple symmetric random walk on
  $\N_{0}$ or $\Z$ of length $n$ starting at $0$. That is, we have
  $\Prob(S_{0}=0) = 1$ as well as
  \begin{align*}
    \Prob(S_k=j-1\mid S_{k-1}=j)=\Prob(S_k=j+1 \mid
    S_{k-1}=j)&=\frac{1}{2} \qquad\text{ for }j\ge 1,\\ 
    \Prob(S_k = 1 \mid S_{k-1}=0) & = 1,
  \end{align*}
  for random walks defined on $\N_{0}$, and
  \[ \Prob(S_k=j-1\mid S_{k-1}=j)=\Prob(S_k=j+1\mid
  S_{k-1}=j)=\frac{1}{2} \qquad\text{ for } j\in \Z \]
  for random walks on $\Z$. Then $(S_{k})_{0\leq k\leq n}$ is said to
  be \emph{admissible} of height $h$, if the random walk stays
  within the interval $[0, h]$ and ends in $h$, i.e.\ $S_{k}\in [0,h]$
  for all $0\leq k\leq n$ and $\Prob(S_{n} = h) = 1$. It is called
  \emph{admissible}, if it is admissible of any height $h \in \N$.

  The probability that a random walk of length $n$ is admissible of
  height $h$ is written as $p_{n}^{(h)}$ and $q_{n}^{(h)}$ for
  random walks on $\N_{0}$ and $\Z$, respectively. Furthermore, the
  probabilities that a random walk is admissible at all are defined as
  $p_{n} := \sum_{h\geq 0} p_{n}^{(h)}$ and $q_{n} := \sum_{h\geq 0}
  q_{n}^{(h)}$ respectively.

  Finally, an \emph{admissible lattice path} is a sequence of integers
  realizing an admissible random walk.
\end{definition}

In a nutshell, this means that an admissible random walk is a non-negative random walk ending in its maximum. The definition is also
visualized in Figure~\ref{fig:admis-LP}, where all admissible lattice
paths of length $5$ are depicted. There are three
admissible lattice paths of height $3$, and one of height $1$ and
$5$, respectively. Note that when considering random walks on $\Z$,
every lattice path has the same probability $2^{-n}$. Admissible
random walks on $\Z$ are enumerated by
sequence \href{http://oeis.org/A167510}{A167510} in \cite{OEIS:2015}.

\begin{figure}[ht]
  \centering
  \begin{tikzpicture}[scale=0.45]
    \draw[->, thick] (-0.5,0) -- (30.5,0);
    \draw[->, thick] (0,-0.5) -- (0,5.5);
    \foreach \x in {1,...,5} {
    \draw[dotted, gray] (0,\x) -- (30.5,\x);
    }
    \draw[->] (0,0) to (1,1);
    \draw[->] (1,1) to (2,2);
    \draw[->] (2,2) to (3,3);
    \draw[->] (3,3) to (4,4);
    \draw[->] (4,4) to (5,5);

    \draw[->] (6,0) to (7,1);
    \draw[->] (7,1) to (8,0);
    \draw[->] (8,0) to (9,1);
    \draw[->] (9,1) to (10,2);
    \draw[->] (10,2) to (11,3);
    
    \draw[->] (12,0) to (13,1); 
    \draw[->] (13,1) to (14,2);
    \draw[->] (14,2) to (15,1);
    \draw[->] (15,1) to (16,2);
    \draw[->] (16,2) to (17,3);

    \draw[->] (18,0) to (19,1); 
    \draw[->] (19,1) to (20,2); 
    \draw[->] (20,2) to (21,3); 
    \draw[->] (21,3) to (22,2); 
    \draw[->] (22,2) to (23,3); 

    \draw[->] (24,0) to (25,1);
    \draw[->] (25,1) to (26,0);
    \draw[->] (26,0) to (27,1);
    \draw[->] (27,1) to (28,0);
    \draw[->] (28,0) to (29,1);

  \end{tikzpicture}
  \caption{Admissible lattice paths of length $5$}
  \label{fig:admis-LP}
\end{figure}
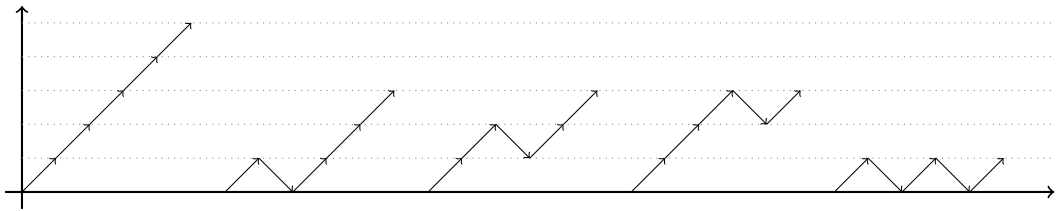

However, in the case of random walks on $\N_{0}$, the probability
depends on the number of visits to $0$: if there are $v$ such visits
(including the initial state), then the path occurs with probability
$2^{-n+v}$. Note that by ``folding down'' (i.e., reflecting about the $x$-axis) some sections between consecutive
visits to $0$, or the section between the last visit and the end, $2^{v}$
lattice paths on $\Z$ can be formed, where the random walk is never farther
away from the start than at the end. We will call such lattice paths
\emph{extremal lattice paths}---and by construction, the number of
extremal lattice paths of length $n$ is given by $p_{n} 2^{n}$. To
illustrate this idea of extremal lattice paths, all paths of this form
of length $3$ are given in Figure~\ref{fig:extremal-LP}.

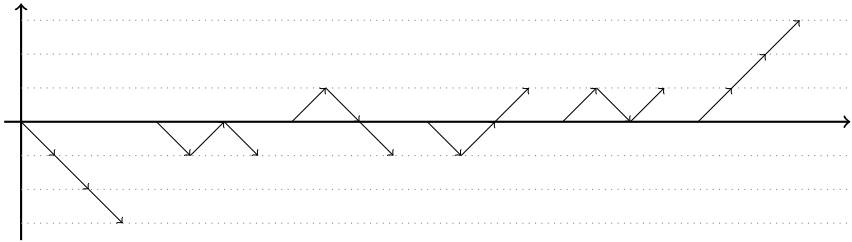
\begin{figure}[ht]
  \centering
  \begin{tikzpicture}[scale=0.45]
    \draw[->, thick] (-0.5,0) -- (24.5,0);
    \draw[->, thick] (0,-3.5) -- (0,3.5);
    \foreach \x in {-3,-2,-1,1,2,3} {
      \draw[dotted, gray] (0,\x) -- (24.5,\x);
    }
    
    \draw[->] (0,0) to (1,-1);
    \draw[->] (1,-1) to (2,-2);
    \draw[->] (2,-2) to (3,-3);

    \draw[->] (4,0) to (5,-1);
    \draw[->] (5,-1) to (6,0);
    \draw[->] (6,0) to (7,-1);
    
    \draw[->] (8,0) to (9,1); 
    \draw[->] (9,1) to (10,0);
    \draw[->] (10,0) to (11,-1);

    \draw[->] (12,0) to (13,-1); 
    \draw[->] (13,-1) to (14,0); 
    \draw[->] (14,0) to (15,1); 

    \draw[->] (16,0) to (17,1);
    \draw[->] (17,1) to (18,0);
    \draw[->] (18,0) to (19,1);

    \draw[->] (20,0) to (21,1);
    \draw[->] (21,1) to (22,2);
    \draw[->] (22,2) to (23,3);
  \end{tikzpicture}
  \caption{Extremal lattice paths of length $3$}
  \label{fig:extremal-LP}
\end{figure}

One of our motivations for investigating admissible random walks
originates from a conjecture in \cite{Zhao:2010:const-mstd}. There, Zhao
introduced the notion of a \emph{bidirectional ballot sequence}:

\begin{definition}[{\cite[Definition~3.1]{Zhao:2010:const-mstd}}]
  A $0$-$1$ sequence is called a \emph{bidirectional ballot
    sequence} if every prefix and suffix contains strictly more $1$'s than
  $0$'s. The number of bidirectional ballot sequences of length $n$ is denoted
  by $B_n$.
\end{definition}

Bidirectional ballot sequences are strongly related to admissible
random walks on $\Z$. In fact, every bidirectional ballot sequence
of length $n+2$ bijectively corresponds to an admissible random walk
of length $n$ on $\Z$: given an admissible random walk, every up-step
corresponds to a $1$, and down-steps correspond to $0$. Adding a $1$
both at the beginning and at the end of the constructed string gives a
bidirectional ballot sequence of length $n+2$.

Therefore, bidirectional ballot walks may also be
seen as lattice paths with unique minimum and maximum.

While we restrict ourselves to simple lattice paths (i.e.\ the path has
steps $\pm 1$), Bousquet-M\'elou and Ponty introduce a more general class of
so-called \emph{culminating paths} in \cite{Bousquet-Ponty:2008:culminating}. Akin to
bidirectional ballot walks, culminating paths are lattice paths with unique minimum and
maximum---however, the lattice path has steps $a$ and $-b$ for fixed $a$, $b > 0$. They
show that the behavior of these paths strongly depends on the drift $a - b$. In
particular, for $a = b = 1$ (i.e.\ for bidirectional ballot walks) they determine
the main term of the asymptotic expansion of $B_{n}$ (cf.\ \cite[Proposition
4.1]{Bousquet-Ponty:2008:culminating}).

In \cite{Zhao:2010:const-mstd}, Zhao also shows that
$B_n=\Theta(2^n/n)$, states (without detailed proof) that $B_n\sim
2^n/(4n)$ and conjectures that 
\begin{equation*}
  \frac{B_n}{2^n}=\frac{1}{4n}+\frac{1}{6n^2}+O\Big(\frac{1}{n^3}\Big).
\end{equation*}

In this paper, we want to give a detailed analysis of the asymptotic
behavior of admissible random walks. By exploiting the bijection
between admissible random walks and bidirectional ballot sequences,
we also prove a stronger version of Zhao's conjecture.

In order to do so, we use a connection between
Chebyshev polynomials and the probabilities
$p_{n}^{(h)}$ and $q_{n}^{(h)}$
(cf.~Proposition~\ref{proposition:random-walk-U} and
Proposition~\ref{proposition:random-walk-T}, respectively), which we
explore in detail in Section~\ref{sec:cheby}. This allows us to determine 
explicit representations of the probabilities $p_{n}$ and
$q_{n}$, which are given in Theorem~\ref{theorem:prob-explicit}. The
analysis of the asymptotic behavior of admissible random 
walks of given length shall focus in particular on the height of these random walks. In this context, we define random variables $H_{n}$ and
$\widetilde H_{n}$ by
\[ \Prob(H_{n} = h) := \frac{p_{n}^{(h)}}{p_{n}},\qquad \Prob(\widetilde
H_{n} = h) := \frac{q_{n}^{(h)}}{q_{n}}.  \]
These random variables model the height of admissible random walks
on $\N_{0}$ and $\Z$, respectively. Besides an asymptotic expansion
for $p_{n}$ and $q_{n}$, we are also interested in the behavior of the
expected height and its variance. The asymptotic analysis of these
expressions, which is based on an approach featuring the Mellin
transform, is carried out in Section~\ref{sec:RW-N} and
Section~\ref{sec:RW-Z}, and the results are given in
Theorem~\ref{thm:asy-NN} and Theorem~\ref{thm:asy-ZZ},
respectively. Finally, Zhao's conjecture is proved in
Corollary~\ref{cor:zhao}.

\section{Chebyshev Polynomials and Random Walks}\label{sec:cheby}

We denote the Chebyshev polynomials of the first and second kind by 
$T_h$ and $U_h$, respectively, i.e.,
\begin{align*}
  T_{h+1}(x)&=2x T_h(x)-T_{h-1}(x) &&\text{ for $h\ge 1$},&T_0(x)&=1,&T_1(x)&=x,\\
  U_{h+1}(x)&=2x U_h(x)-U_{h-1}(x) &&\text{ for $h\ge 1$},&U_0(x)&=1,&U_1(x)&=2x.
\end{align*}

In the following propositions, we show that these polynomials occur
when analyzing admissible random walks. As usual, the notation $[z^{n}]f(z)$ denotes the
coefficient of $z^{n}$ in the series expansion of $f(z)$.

\begin{proposition}\label{proposition:random-walk-U}
  The probability that a simple symmetric random walk $(S_{k})_{0\leq
    k\leq n}$ of length $n$ on $\Z$ is admissible of height $h$ is
  \begin{equation}\label{eq:U-probability}
    q_{n}^{(h)} = \Prob( 0\le S_0, S_1, \ldots, S_n \le h \text{ and }S_n=h )=2 [z^{n+1}] \frac{1}{U_{h+1}(1/z)}
  \end{equation}
  for $h\ge 0$ and $n\ge 0$.
\end{proposition}
\begin{proof}  
  We consider the $(h+1)\times(h+1)$ transfer matrix
  \begin{equation*}
    M_h=
    \begin{pmatrix}
      0&\frac12&0&\cdots&\cdots&0\\
      \frac12&0&\frac12&\cdots&\cdots&0\\
      0&\frac12&0&\ddots&&0\\
      \vdots&\vdots&\ddots&\ddots&\ddots&\vdots\\
      \vdots&\vdots & & \ddots &   \ddots  &\frac{1}{2} \\
      0&0&0&\ldots&\frac{1}{2}&0
    \end{pmatrix},
  \end{equation*}
which has the following simple yet useful property: if $w_{n,k}^{(h)}$ is the probability that $0 \leq S_0,S_1,\ldots,S_n \leq h$ and $S_n = k$, then the following recursion for the vectors $w_n^{(h)} = (w_{n,0}^{(h)},w_{n,1}^{(h)},\ldots,w_{n,h}^{(h)})$ holds:
$$w_n^{(h)} \cdot M_h = w_{n+1}^{(h)},$$
so $w_n^{(h)} = w_0^{(h)} \cdot M_h^n$. The initial vector is $w_0^{(h)} = e_0=(1,0,\ldots,0)$. Since we also want that $S_n = h$, we multiply by the vector $e_h =(0,\ldots,0,1)^\top$ at the end to extract only the last entry $w_{n,h}^{(h)}$. This yields the generating function
  \begin{equation*}
    \sum_{n\ge 0}q_{n}^{(h)}z^n=\sum_{n\ge 0}e_0 M_h^n e_h z^n = e_0 (I-zM_h)^{-1} e_h.
  \end{equation*}
  Cramer's rule yields
  \[ \sum_{n\geq 0} q_{n}^{(h)} z^{n} = \frac{z^{h} 2^{-h}}{\det(I -
    zM_{h})}.  \]
 The determinant of $I-zM_h$ can be computed recursively in $h$ by
  means of row expansion, see (for instance) \cite[p.97]{Aigner:2007:course-enumeration}:
  \[ \det(I - z M_{h+2}) = \det(I - zM_{h+1}) - \frac{z^{2}}{4} \det(I
  - zM_{h}).\] 
  Comparing this with the recursion for the Chebyshev polynomials and checking the initial values, we find that $\frac{2^{h+1}
    \det(I - zM_{h})}{z^{h+1}} = U_{h+1}(1/z)$. Therefore, we obtain
  \begin{equation*}
    \sum_{n\ge 0}q_{n}^{(h)}z^n=\frac{2}{zU_{h+1}(1/z)}
  \end{equation*}
  from which \eqref{eq:U-probability} follows by extracting the coefficient of $z^n$.
\end{proof}

An analogous statement holds for admissible random walks on $\N_{0}$
with the sole difference that in this case, the Chebyshev polynomials
of the first kind occur.

\begin{proposition}\label{proposition:random-walk-T}
  The probability that a random walk $(S_{k})_{0\leq k \leq n}$ of
  length $n$ on $\N_{0}$ is admissible of height $h$ is given by
  \begin{equation}\label{eq:T-probability}
    p_{n}^{(h)} = \Prob( 0\le S_0, S_1, \ldots, S_n \le h \text{ and }S_n=h )=2 [z^{n+1}] \frac{1}{T_{h+1}(1/z)}
  \end{equation}
  for $h\ge 0$ and $n\ge 1$. 
\end{proposition}
\begin{proof}
  For random walks with a reflective barrier at $0$, the $(h+1)\times
  (h+1)$ transfer matrix has the form
  \[ \widetilde M_{h}=
  \begin{pmatrix}
    0& 1 &0&\cdots&\cdots&0\\
    \frac12&0&\frac12&\cdots&\cdots&0\\
    0&\frac12&0&\ddots&&0\\
    \vdots&\vdots&\ddots&\ddots&\ddots&\vdots\\
    \vdots&\vdots & & \ddots &   \ddots  &\frac{1}{2} \\
    0&0&0&\ldots&\frac{1}{2}&0
  \end{pmatrix}. \]
  By the same approach involving Cramer's rule as in the proof of
  Proposition~\ref{proposition:random-walk-U}, we find the generating
  function
  \[ \sum_{n\geq 0} p_{n}^{(h)} z^{n} = e_{0} (I - z\widetilde
  M_{h})^{-1} e_{h} = \frac{z^{h} 2^{1-h}}{\det(I-z\widetilde
    M_{h})}, \]
  where we have the recursion
  \[ \det(I - z\widetilde M_{h+2}) = \det(I - z\widetilde M_{h+1}) -
  \frac{z^{2}}{4} \det(I - z\widetilde M_{h})  \]
  for the determinant of $I-z\widetilde M_{h}$. Finally,
  (\ref{eq:T-probability}) follows from $\frac{2^{h-1} \det(I-z\widetilde
    M_{h-1})}{z^{h}} = T_{h}(1/z)$, which can be proved again by verifying that
the same recursion holds for the Cheby\-shev-$T$ polynomials and that the initial values agree.
\end{proof}
\begin{remark}
  The coefficients of $\frac{1}{T_{h}(1/z)}$ have also been studied in
  \cite{Jiu-Moll-Vignat:2014:euler-poly}. There, the case of fixed $h$
  is investigated, whereas we mostly focus on the asymptotic behavior of
  $\sum_{h\geq 0} p_{n}^{(h)}$ for $n\to\infty$.
\end{remark}

Using the results from Proposition~\ref{proposition:random-walk-U} and
Proposition~\ref{proposition:random-walk-T}, we may give explicit representations
of the probabilities $p_{n}^{(h)}$ and $q_{n}^{(h)}$ by investigating
the Chebyshev polynomials thoroughly.

\begin{remark}[Iverson's notation]
  We use the Iversonian notation 
$$\llbracket \mathit{expr} \rrbracket = \begin{cases} 1& \text{if $\mathit{expr}$ is true,} \\ 0 & \text{otherwise,}\end{cases}$$
 popularized in \cite[Chapter 2]{Graham-Knuth-Patashnik:1994}.
\end{remark}

In the following theorem and throughout the rest of the paper, $m$ will denote a half-integer, i.e., $m \in \frac12 \N = \{\frac12, 1, \frac32, 2, \ldots\}$. While this convention may seem unusual, it simplifies many of our formulas and is therefore convenient for calculations.

\begin{theorem}\label{theorem:prob-explicit}
  With $\tau_{h,k} := (h+1)(2k+1)/2$ and $\upsilon_{h,k} :=
  (h+2)(2k+1)/2$, we have
  \begin{align}
    p_{2m-1}^{(h)} &= \frac{4}{4^{m}} \sum_{k\geq 0} (-1)^{k}
    \frac{\tau_{h,k}}{m} \binom{2m}{m-\tau_{h,k}} \cdot \llbracket h+1
    \equiv 2m \bmod 2 \rrbracket, \label{eq:prob-expl-T}\\
    q_{2m-2}^{(h)} &= \frac{4}{4^{m}} \sum_{k\geq 0} \frac{2
                     \upsilon_{h,k}^{2} - m}{(2m-1) m}
                     \binom{2m}{m-\upsilon_{h,k}} \cdot \llbracket h
                     \equiv 2m \bmod 2\rrbracket \label{eq:prob-expl-U}
  \end{align}
  for $h \geq 0$ and half-integers $m\in \frac{1}{2} \N$ with $m\geq 1$.
\end{theorem}
\begin{proof}
  We begin with the analysis of $p_{n}^{(h)}$. The
  probabilities are related to the Chebyshev-$T$ polynomials by
  Proposition~\ref{proposition:random-walk-T}. It is a 
  well-known fact (cf.\ \cite[22:3:3]{Oldham-Spanier:1987:atlas-funct}) that these polynomials
  have the explicit  representation
  \begin{equation*}
    T_{h}(x) = \frac{(x - \sqrt{x^{2} - 1}\,)^{h} + (x + \sqrt{x^{2} -1}\,)^{h}}{2},
  \end{equation*}
  which immediately yields
  \begin{equation}\label{eq:chebyshev-T-def-Y}
  \frac{1}{T_{h}(1/z)} = z^{h} \frac{2}{(1 - \sqrt{1 -
      z^{2}}\,)^{h} + (1 + \sqrt{1 - z^{2}}\,)^{h}} =: z^{h} Y(z^{2}).
  \end{equation}
  By applying Cauchy's integral formula, we obtain the coefficients of
  the factor $Y(t)$ encountered in \eqref{eq:chebyshev-T-def-Y}. We
  choose a sufficiently small circle around $0$ as the integration
  contour $\gamma$. Thus, we get 
  \begin{align*}
    [t^{n}] Y(t) & = [t^{n}]\frac{2}{(1 - \sqrt{1 - t}\,)^{h} + (1
                   + \sqrt{1 - t}\,)^{h}} \\
                 & = \frac{1}{2\pi i}
                   \oint_{\gamma} \frac{2}{(1 - \sqrt{1-t}\,)^{h} + (1 +
                   \sqrt{1-t}\,)^{h}} \cdot \frac{1}{t^{n+1}}~dt.
  \end{align*}
  We want to simplify the expression $\sqrt{1-t}$ in this
  integral. This can be achieved by the substitution $t =
  \frac{4u}{(1+u)^{2}}$, which gives us $dt = (1-u)\cdot
  \frac{4}{(1+u)^{3}}~du$ and $\sqrt{1-t} =
  \frac{1-u}{1+u}$. Also, the new integration contour is
  $\tilde\gamma$, which is still a contour that winds around the
  origin once. Then, again by Cauchy's integral formula, we obtain
  \begin{align*}
    [t^{n}] Y(t) & = \frac{1}{2\pi i} \oint_{\tilde\gamma} (1-u)
                   \frac{(1+u)^{2n+h-1}}{2^{2n+h-1} (1+u^{h})}\cdot
                   \frac{1}{u^{n+1}}~du \\
                 & = [u^{n}] (1-u)
                   \frac{(1+u)^{2n+h-1}}{2^{2n+h-1} (1+u^{h})}.
  \end{align*}
Expanding the factor $\frac{(1+u)^{2n+h-1}}{1+u^{h}}$ into a series
with the help of the geometric series and the binomial theorem yields
\[ \frac{(1+u)^{2n+h-1}}{1 + u^{h}} = \sum_{k\geq 0} (-1)^{k}
u^{kh} (1 + u)^{2n+h-1} = \sum_{k\geq 0} (-1)^{k} u^{kh}
\sum_{j=0}^{2n+h-1} \binom{2n+h-1}{j} u^{j},  \]
and therefore
\[ [u^{\ell}]\frac{(1+u)^{2n+h-1}}{1+u^{h}} = \sum_{k\geq
  0} (-1)^{k} \binom{2n+h-1}{\ell - hk}.  \]
This allows us to expand the expression encountered before, that is
\begin{align*}
  [t^{n}] Y(t) & = [u^{n}] (1-u) \frac{(1+u)^{2n+h-1}}{2^{2n+h-1}
  (1+u^{h})} \\
               & = \frac{1}{2^{2n+h-1}} \sum_{k\geq
                      0}(-1)^{k} \left[\binom{2n+h-1}{n-hk} -
                      \binom{2n+h-1}{n-hk-1}\right].
\end{align*}
Using the binomial identity
\[ \binom{N-1}{\alpha} - \binom{N-1}{\alpha-1} = \frac{N-2\alpha}{N}
\binom{N}{\alpha}, \]
the expression above can be simplified so that, together
with~\eqref{eq:chebyshev-T-def-Y}, we find 
\begin{equation*}
  \frac{1}{T_{h}(1/z)} = 2 \sum_{n\geq 0}
  \left(\frac{z}{2}\right)^{2n+h} \sum_{k\geq 0} (-1)^{k}
  \frac{2hk + h}{2n+h} \binom{2n+h}{n-hk}.
\end{equation*}
By plugging this into \eqref{eq:T-probability}, we obtain
\begin{align*}
  p_{n}^{(h)} & = 2 [z^{n+1}]\frac{1}{T_{h+1}(1/z)}
  \\
              & =
  4 [z^{n+1}] \sum_{\ell\geq 0}
  \left(\frac{z}{2}\right)^{2\ell+h+1} \sum_{k\geq 0} (-1)^{k}
  \frac{2(h+1)k + h+1}{2\ell+h+1}\binom{2\ell+h+1}{\ell -
                (h+1)k}\\
  & = \frac{1}{2^{h-1}} [z^{n-h}] \sum_{\ell\geq 0}
    \left(\frac{z}{2}\right)^{2\ell} \sum_{k\geq 0} (-1)^{k}
    \frac{2(h+1)k + h + 1}{2\ell + h + 1} \binom{2\ell+h+1}{\ell-(h+1)k}.
\end{align*}
Combinatorially, it is clear that $p_{n}^{(h)} = 0$ for $n$
and $h$ of different parity, as only heights of the same
parity as the length can be reached by a random walk starting at the origin. This can
also be observed in the representation above. Assuming $n\equiv h\bmod
2$, we can write $n-h = 2\ell$ or equivalently
$\frac{n-h}{2} = \ell$. This gives us
\begin{align*} 
  p_{n}^{(h)} & = \frac{1}{2^{n-1}} \sum_{k\geq 0} (-1)^{k}
                \frac{2(h+1)k + h + 1}{n+1}
                \binom{n+1}{\frac{n-h}{2} - (h+1)k}\\
              & = \frac{1}{2^{n-1}} \sum_{k\geq 0} (-1)^{k}
                \frac{(h+1)(2k+1)}{n+1} \binom{n+1}{\frac{n+1}{2} -
                \frac{1}{2} (h+1)(2k+1)}.
\end{align*}
Substituting $n = 2m-1$ with a half-integer $m \in \frac{1}{2} \N$ such
that $h+1 \equiv 2m\bmod 2$, and recalling that $\tau_{h,k} =
(h+1)(2k+1)/2$, the representation in \eqref{eq:prob-expl-T} is
proved.

For the second part, we consider the explicit
representation
\begin{equation*}
  U_{h}(x) = \frac{(x + \sqrt{x^{2} - 1}\,)^{h+1} - (x - \sqrt{x^{2} -
      1}\,)^{h+1}}{2\sqrt{x^{2} - 1}}
\end{equation*}
of the Chebyshev-$U$ polynomials, which is equivalent to
\[ \frac{1}{U_{h}(1/z)} = z^{h} \frac{2 \sqrt{1 -
    z^{2}}}{(1+\sqrt{1-z^{2}}\,)^{h+1} - (1-\sqrt{1-z^{2}}\,)^{h+1}}. 
\]
Formula~\eqref{eq:prob-expl-U} is now obtained in the same way as~\eqref{eq:prob-expl-T}. 

\end{proof}

With explicit formulae for the probabilities $p_{n}^{(h)}$ and
$q_{n}^{(h)}$, we can start to work towards the analysis of
the asymptotic behavior of admissible random walks.

\section{Admissible Random Walks on \texorpdfstring{$\N_{0}$}{N}}\label{sec:RW-N}

In this section, we begin to develop the tools required for a precise
analysis of the asymptotic behavior of admissible random walks on
$\N_{0}$. 

Recalling the result of Theorem~\ref{theorem:prob-explicit},
we find that in the half-integer representation $p_{2m-1}^{(h)}$, the
shifted central binomial coefficient $\binom{2m}{m-\tau_{h,k}}$
appears. Hence, for the purpose of obtaining an expansion for
$p_{2m-1} = \sum_{h \geq 0} p_{2m-1}^{(h)}$, analyzing the asymptotics of binomial coefficients in the central region is necessary. In the following, we will work a lot with asymptotic expansions. The notation
$$f(n) \sim \sum_{\ell=-L}^{\infty} a_{\ell} n^{-\ell}$$
(as $n \to \infty$) is understood to mean
$$f(n) = \sum_{\ell=-L}^{R-1} a_{\ell} n^{-\ell} + O(n^{-R})$$
for all integers $R > -L$, even if the series does not converge. Likewise, an asymptotic expansion in two variables given by
$$f(\alpha,n) \sim \sum_{\ell=-L}^{\infty} \sum_{j=0}^{J(\ell)} b_{\ell j} \frac{\alpha^j}{n^{\ell}}$$
is to be understood as
$$f(\alpha,n) = \sum_{\ell=-L}^{R-1} \sum_{j=0}^{J(\ell)} b_{\ell j} \frac{\alpha^j}{n^{\ell}} + O(\alpha^{J(R)}n^{-R})$$
for all $R > -L$.

\begin{lemma}\label{lemma:central-binom-asy}
  For $n\in \frac{1}{2}\N$ and $|\alpha| \leq n^{2/3}$ such that $n - \alpha \in   \N$, we have
  \[ \binom{2n}{n-\alpha} \sim \frac{4^{n}}{\sqrt{n\pi}} \exp\left(-
    \frac{\alpha^{2}}{n}\right) \cdot S(\alpha, n)  \]
  with $S(\alpha, n) := \sum_{\ell, j\geq 0} c_{\ell j}
  \frac{\alpha^{2j}}{n^{\ell}}$ and
  \begin{align}
    c_{\ell j} & = [\alpha^{2j} n^{-\ell}]\bigg(\sum_{r\geq 0}
                 \frac{d_{r}}{(2n)^{r}}\bigg) \bigg(\sum_{r\geq
                 0} \frac{(-1)^{r} d_{r}}{(n+\alpha)^{r}}\bigg)
                 \bigg(\sum_{r\geq 0} \frac{ (-1)^{r}
                 d_{j}}{(n-\alpha)^{r}}\bigg)\label{eq:binom-coef}\\ 
               & \qquad \times \bigg(\sum_{r\geq 0} (-1)^{r}
                 \binom{-1/2}{r}
                 \frac{\alpha^{2r}}{n^{2r}}\bigg)\bigg(\sum_{r\geq 0}
                 \frac{1}{r!}
                 \frac{\alpha^{4r}}{n^{3r}} \bigg[ \sum_{t \geq 0}
                 \frac{-1}{(t + 2)(2t + 3)}
                 \frac{\alpha^{2t}}{n^{2t}} \bigg]^{r}\bigg),\notag
  \end{align}
  where the coefficients $d_{r}$ come from the higher-order Stirling
  approximation of the factorial,
  cf.~\eqref{eq:stirling}. Additionally, the estimate 
  \[ S(\alpha, n) = 1 + O\bigg(\frac{1 + |\alpha|}{n}\bigg)  \]
  holds for $|\alpha| \leq n^{2/3}$ and we know that $c_{00} = 1$
  as well as $c_{\ell j} = 0$ if $j > \frac{2}{3} \ell$.
  
  If $|\alpha| > n^{2/3}$, the term
  \[ \binom{2n}{n-\alpha} / 4^{n} = O(\exp(-n^{1/3})) \] 
  decays faster than any power of $n$.
\end{lemma}
\begin{proof}
  We begin by recalling the higher-order Stirling approximation (cf.\
  \cite[p.~760]{Flajolet-Sedgewick:ta:analy}) 
  \begin{equation}\label{eq:stirling}
    n! \sim \sqrt{2\pi n} \left(\frac{n}{e}\right)^{n} \bigg(\sum_{j\geq
        0} \frac{d_{j}}{n^{j}}\bigg).
  \end{equation}
  An explicit representation of the coefficients $d_{j}$ can be
  found in \cite{Nemes:2010:coef-factorial}. From the logarithmic
  representation of the factorial (see
  \cite[p.~766]{Flajolet-Sedgewick:ta:analy}), the expansion
  \begin{equation}\label{eq:stirling-reciprocal}
    \frac{1}{n!} \sim \frac{1}{\sqrt{2\pi n}} \Big(\frac{e}{n}\Big)^{n}
    \bigg(\sum_{j\geq 0} \frac{(-1)^{j} d_{j}}{n^{j}}\bigg)
  \end{equation}
  for the reciprocal factorial follows.

  Let us assume $|\alpha| \leq n^{2/3}$. Then, by applying
  \eqref{eq:stirling} and \eqref{eq:stirling-reciprocal} to the
  shifted central binomial coefficient, we obtain
  \begin{align*}
    \binom{2n}{n-\alpha} 
    & = \frac{(2n)!}{(n-\alpha)!\, (n+\alpha)!} \\
    & = \frac{1}{\sqrt{n \pi}} \bigg(1 -
      \frac{\alpha^{2}}{n^{2}}\bigg)^{-1/2}
      \frac{(2n)^{2n}}{(n+\alpha)^{n+\alpha} (n-\alpha)^{n-\alpha}}\\
    & \qquad\qquad \times \bigg(\sum_{r\geq 0}
      \frac{d_{r}}{(2n)^{r}}\bigg) \bigg(\sum_{r\geq 0}
      \frac{(-1)^{r} d_{r}}{(n+\alpha)^{r}}\bigg) \bigg(\sum_{r\geq
      0} \frac{(-1)^{r} d_{r}}{(n-\alpha)^{r}}\bigg).
  \end{align*}
  The factor $\big(1 - \frac{\alpha^{2}}{n^{2}}\big)^{-1/2}$ can be
  expanded as a binomial series, resulting in
  \[ \Big(1 - \frac{\alpha^{2}}{n^{2}}\Big)^{-1/2} = \sum_{r\geq 0}
  (-1)^{r} \binom{-1/2}{r} \frac{\alpha^{2r}}{n^{2r}}.  \]
  
  The remaining factor is handled by means of the identity $n^{n} =
  \exp(n\log n)$, which leads to
  \begin{align*}
    \frac{(2n)^{2n}}{(n+\alpha)^{n+\alpha} (n-\alpha)^{n-\alpha}} 
    & = \exp(2n \log(2n) -(n+\alpha)\log(n+\alpha) - (n-\alpha)\log(n-\alpha))\\ 
    & = \exp(2n\log2 + 2n\log n - (n+\alpha)(\log n + \log(1 + \alpha/n)) \\
    & \qquad\qquad\qquad\qquad {} - (n-\alpha)(\log n + \log(1 - \alpha/n))) \\
    & = 4^{n}  \exp(\alpha\log(1 - \alpha/n) - \alpha\log(1 + \alpha/n) \\
    & \qquad\qquad\qquad\qquad {} - n\log(1 - \alpha/n) - n\log(1 +
      \alpha/n)).
  \end{align*}
By expanding the logarithm into a power series, we
  can simplify this expression to
  \begin{align*}
    \frac{(2n)^{2n}}{(n+\alpha)^{n+\alpha} (n-\alpha)^{n-\alpha}} 
    & = 4^{n} \exp\bigg(2 \bigg[- \sum_{t\geq 0}
      \frac{1}{2t+1} \frac{\alpha^{2t+2}}{n^{2t+1}} + \sum_{t\geq 0}
      \frac{1}{2t+2} \frac{\alpha^{2t+2}}{n^{2t+1}}\bigg]\bigg) \\
    & = 4^{n}  \exp\Big(-\frac{\alpha^{2}}{n}\Big) 
      \exp\bigg(- \frac{\alpha^{4}}{n^{3}} \sum_{t\geq 0}
      \frac{1}{(t+2)(2t+3)} \frac{\alpha^{2t}}{n^{2t}} \bigg) \\
    & = 4^{n} \exp\Big(-\frac{\alpha^{2}}{n}\Big) 
      \bigg(\sum_{r\geq 0} \frac{1}{r!} \frac{\alpha^{4r}}{n^{3r}}
      \bigg[\sum_{t\geq 0} \frac{-1}{(t+2)(2t+3)}\frac{\alpha^{2t}}{n^{2t}}\bigg]^{r}\bigg).
  \end{align*}
  We also use
  \[ \frac{1}{n\pm\alpha} = \frac{1}{n} \frac{1}{1 \pm
    \frac{\alpha}{n}} = \frac{1}{n} \sum_{r\geq 0} \left(\mp
    \frac{\alpha}{n}\right)^{r}.  \]
  By the symmetry of the binomial coefficient, the resulting
  asymptotic expansion has to be symmetric in $\alpha$.
  Assembling all these expansions yields the asymptotic formula
  \[ \binom{2n}{n-\alpha} \sim \frac{4^{n}}{\sqrt{n\pi}}
  \exp\Big(-\frac{\alpha^{2}}{n}\Big) \cdot S(\alpha, n),  \]
  where $S(\alpha, n)$ is defined as in the statement of the lemma. 

  Note that $d_{0} = 1$, and thus the first summand of the
  series in~\eqref{eq:binom-coef} is $1$---which gives $c_{00} =
  1$. Summands where the exponent of $\alpha$ exceeds the exponent of $1/n$ only occur in
  the last series, with the maximal difference being induced by
  $\alpha^{4r}/n^{3r}$. Thus, if $j > \frac{2}{3} \ell$, we have
  $c_{\ell j} = 0$. Together with $|\alpha| \leq n^{2/3}$, this
  implies the estimate for $S(\alpha, n)$. 

  For $|\alpha| > n^{2/3}$, we can use the monotonicity of the
  binomial coefficient to obtain
  \[ \binom{2n}{n-\alpha} \leq \binom{2n}{n - \lceil n^{2/3} \rceil},   \]
  for which the exponential factor ensures fast decay,
  \[ \exp\Big(-\frac{\lceil n^{2/3}\rceil^{2}}{n}\Big) =
  O\left(\exp\left(-n^{1/3}\right)\right),  \]
  and as everything else is of polynomial growth, the statement of the
  lemma follows.
\end{proof}

Now that we have an asymptotic expansion for the shifted central
binomial coefficient, let us look at our explicit formula in~\eqref{eq:prob-expl-T} again: we have
\begin{equation*}
    p_{2m-1}^{(h)} = \frac{4}{4^{m}} \sum_{k\geq 0} (-1)^{k}
    \frac{\tau_{h,k}}{m} \binom{2m}{m-\tau_{h,k}} \cdot \llbracket h+1
    \equiv 2m \bmod 2 \rrbracket,
\end{equation*}
where $\tau_{h,k} = (h+1)(2k+1)/2$. Therefore, the total probability for a random walk of length $2m-1$ on $\N_0$ to be admissible is given by
\begin{equation*}
p_{2m-1} = \sum_{h \geq 0}  p_{2m-1}^{(h)} = \frac{4}{4^{m}} \sum_{\substack{h,k\geq 0 \\ h+1\equiv
    2m\bmod 2}} (-1)^{k} \frac{\tau_{h,k}}{m} \binom{2m}{m-\tau_{h,k}}.
\end{equation*}
The terms where $\tau_{h,k} > m^{2/3}$ can be neglected in view of the last statement in Lemma~\ref{lemma:central-binom-asy}, as their total contribution decays faster than any power of $m$: note that there are only $O(m^2)$ such terms (trivially, $h$, $k \leq m$), each of which contributes $O(m \exp(-m^{1/3}))$ to the sum. For all other values of $h$ and $k$, we can replace the binomial coefficient by its asymptotic expansion. This gives us, for any $L > 0$,
\begin{align*}
  p_{2m-1} &= \frac{4}{\sqrt{m\pi}}
  \hspace{-1em}\sum_{\substack{h,k\geq 0, \tau_{h,k} \leq m^{2/3} \\ h+1\equiv
    2m\bmod 2}} \hspace{-1em} (-1)^{k}
  \frac{\tau_{h,k}}{m} \exp\Big(-\frac{\tau_{h,k}^{2}}{m}\Big) \sum_{\ell = 0}^{L-1} \sum_{j \geq 0}
  c_{\ell j}  \frac{\tau_{h,k}^{2j}}{m^{\ell}} \\
&\quad + O \bigg( \frac{1}{\sqrt{m}} \sum_{\substack{h,k\geq 0, \tau_{h,k} \leq m^{2/3} \\ h+1\equiv
    2m\bmod 2}} \frac{\tau_{h,k}^{2J(L)+1}}{m^{L+1}}\bigg),
\end{align*}
where $J(L) \leq \frac23L$ since $c_{\ell j} = 0$ for $j > \frac23\ell$. Since the sum clearly contains $O(m^{4/3})$ terms, the error is at most $O(m^{-1/2+4/3+2/3(2J(L)+1)-(L+1)}) = O(m^{1/2-L/9})$. The exponent can be made arbitrarily small by choosing $L$ accordingly. Finally, if we extend the sum to the full range (all integers $h$, $k \geq 0$ such that $h+1 \equiv 2m \bmod 2$) again, we only get another error term of order $O(\exp(-m^{1/3}))$, which can be neglected. In summary, we have
\begin{equation}\label{eq:prob-T-asy}
  p_{2m-1} \sim \frac{4}{\sqrt{m\pi}}
  \hspace{-1em}\sum_{\substack{h,k\geq 0 \\ h+1\equiv
    2m\bmod 2}} \hspace{-1em} (-1)^{k}
  \frac{\tau_{h,k}}{m} \exp\Big(-\frac{\tau_{h,k}^{2}}{m}\Big)
   \sum_{\ell, j\geq 0} c_{\ell j}
  \frac{\tau_{h,k}^{2j}}{m^{\ell}}.
\end{equation}
This sum can be analyzed with the help of the Mellin transform and the
converse mapping theorem (cf.\ \cite{Flajolet-Gourdon-Dumas:1995:mellin}).
In order to follow this approach, we will investigate those
terms in~\eqref{eq:prob-T-asy} whose growth is not obvious more precisely. That
is, we will focus on the contribution of terms of the form
\[ \sum_{\substack{h,k\geq 0 \\ h+1\equiv 2m\bmod 2}} \hspace{-1em}(-1)^{k}
\tau_{h,k}^{2j+1} \exp\Big(-\frac{\tau_{h,k}^{2}}{m}\Big). \]

We are also interested in the expected height and the corresponding
variance and higher moments of admissible random walks. Asymptotic expansions for these
can be obtained by analyzing moments of the random variable $H_{n}$
with $\Prob(H_{n} = h) := \frac{p_{n}^{(h)}}{p_{n}}$, as stated in the
introduction. For the sake of convenience, let us consider the $r$-th
shifted moment $\E(H_{2m-1} + 1)^{r}$. We know
\[ \E(H_{2m-1} + 1)^{r} = \sum_{h\geq 0} (h+1)^{r} \Prob(H_{2m-1} =
h) = \frac{\sum_{h\geq 0} (h+1)^{r} p_{2m-1}^{(h)}}{p_{2m-1}}.  \]
The asymptotic behavior of the denominator is related to the behavior
of the sum from above---and fortunately, the behavior of the numerator
is related to the behavior of the very similar sum
\[ \sum_{\substack{h,k\geq 0 \\ h+1\equiv 2m\bmod 2}} (-1)^{k}
\tau_{h,k}^{2j+1} (h+1)^{r} \exp\Big(-\frac{\tau_{h,k}^{2}}{m}\Big). \]
The following lemma analyzes sums of this structure asymptotically.

\begin{lemma}\label{lemma:summands-T}
  Let $j$, $r\in \N_{0}$. Then we have
  \begin{multline}\label{eq:asymptotic-summands-T}
    \sum_{\substack{h,k\geq 0 \\ h+1\equiv 2m\bmod 2}} (-1)^{k}
    \tau_{h,k}^{2j+1} (h+1)^{r}
    \exp\Big(-\frac{\tau_{h,k}^{2}}{m}\Big) \\ = 2^{r-1}
    \Gamma\Big(j+1+\frac{r}{2}\Big) \beta(r+1) m^{j+1+r/2} + O(m^{-K})
  \end{multline}
  for any fixed $K > 0$, where $\beta(\,\cdot\,)$ denotes the
  Dirichlet beta function.
\end{lemma}

\begin{remark}
The Dirichlet beta function is also often called Catalan beta function, and it is defined by
$$\beta(s) = \sum_{k=0}^{\infty} \frac{(-1)^k}{(2k+1)^s}.$$
It can be expressed in terms of the Hurwitz zeta function as $\beta(s) = 4^{-s}(\zeta(s, 1/4) - \zeta(s, 3/4))$. Amongst many other interesting properties, it satisfies the zeta-like functional equation (cf.\ \cite[3:5:2]{Oldham-Spanier:1987:atlas-funct})
$$\beta(1-s) = (\pi/2)^{-s} \sin(\pi s/2) \Gamma(s) \beta(s),$$
which also implies that $\beta(s)$ has zeros at all negative odd integers.
\end{remark}

\begin{proof}[Proof of Lemma~\ref{lemma:summands-T}]
If we substitute $m = x^{-2}$, the left-hand side of \eqref{eq:asymptotic-summands-T} becomes
  \[ f(x) := \hspace{-1em}\sum_{\substack{h,k\geq 0 \\ h+1\equiv 2m\bmod 2}} \hspace{-1em} (-1)^{k}
  \tau_{h,k}^{2j+1} (h+1)^{r} \exp(-\tau_{h,k}^{2} x^{2}) .  \]
This is a typical example of a harmonic sum, cf.\ \cite[\S 3]{Flajolet-Gourdon-Dumas:1995:mellin}, and the Mellin transform can be applied to obtain its asymptotic behaviour. First of all, it is well-known that the Mellin transform of a harmonic sum of the form $f(x) = \sum_{k \geq 1} a_k g(b_k x)$ can be factored as $\sum_{k \geq 1} a_k b_k^{-s} g^*(s)$ \cite[Lemma
  2]{Flajolet-Gourdon-Dumas:1995:mellin}, provided that the half-plane of absolute convergence of the Dirichlet series $\Lambda(s) = \sum_{k \geq 1} a_k b_k^{-s}$ has non-empty intersection with the fundamental strip of the Mellin transform $g^*$ of the base function $g$. In this particular case, the Dirichlet series is
  \[ \Lambda(s) := \hspace{-1em} \sum_{\substack{h,k\geq 0 \\ h+1\equiv 2m\bmod 2}}
  \hspace{-1em} (-1)^{k} \tau_{h,k}^{2j+1-s} (h+1)^{r}, \] 
  and the base function is $g(x) = \exp(-x^{2})$, with Mellin transform $g^{*}(s) = \frac{1}{2} \Gamma\big(\frac{s}{2}\big)$ and
fundamental strip $\langle 0, \infty\rangle$.


Now we simplify the Dirichlet series.
For $s\in \C$ with $\Re(s) >  2j + 2 +r$, the sum
  \[ \Lambda(s) = 2^{s - (2j+1)} \hspace{-1em}\sum_{\substack{h,k\geq 0 \\ h+1\equiv
    2m\bmod 2}} \hspace{-1em} (-1)^{k} (h+1)^{2j+1+r-s}
  (2k+1)^{2j+1-s} \]
  converges absolutely because it is dominated by the
  zeta function. In view of the definition of the $\beta$ function, this simplifies to
  \begin{align*}
    \Lambda(s) 
               & = 2^{s-(2j+1)} \beta(s-(2j+1)) \kappa_{2m}(s - (2j+1+r)),
  \end{align*}
  where $\kappa_{2m}(s)$ depends on the parity of $2m$. We find
  \[ \kappa_{2m}(s) = \hspace{-1em} \sum_{\substack{h\geq 0 \\ h+1\equiv 2m\bmod 2}}  \hspace{-1em} (h+1)^{-s}
 = \begin{cases} 2^{-s} \zeta(s) & \text{ for } m\in\N, \\ (1-2^{-s})
   \zeta(s) & \text{ for } m\not\in\N. 
  \end{cases}  \]
Thus, the Mellin transform of $f$ is
  \[ f^{*}(s) = \Lambda(s) g^{*}(s) = \frac{1}{2}
  \Gamma\Big(\frac{s}{2}\Big) 2^{s - (2j+1)} \beta(s- (2j+1))
  \kappa_{2m}(s - (2j+1+r)).  \]
  By the converse mapping theorem (see
  \cite[Theorem~4]{Flajolet-Gourdon-Dumas:1995:mellin}), the
  asymptotic growth of $f(x)$ for $x\to 0$ can be found by considering the 
analytic continuation of $f^{*}(s)$ further to the left of the complex   plane and investigating its poles. The
  theorem may be applied because $\Lambda(s)$ has polynomial growth
  and $\Gamma(s/2)$ decays exponentially along vertical lines of the
  complex plane.

  We find that $f^{*}(s)$ has a simple pole at $s = 2j+2+r$, which
  comes from the zeta function in the definition of $\kappa_{2m}$. There
  are no other poles: $\beta$ is an entire function, and the
  poles of $\Gamma$ cancel against the zeros of $\beta$ (at all odd negative integers, see the earlier remark).

  The asymptotic contribution from the pole of $f^{*}$ is
  \begin{align*} 
    \Res(f^{*}, s=2j+2+r)\cdot x^{-(2j+2+r)}
    &= \frac{1}{2} \Gamma\Big(j+1 + \frac{r}{2}\Big) 2^{r+1}
       \beta(r+1)  \frac{1}{2} x^{-(2j+2+r)} \\
    & = 2^{r-1}\Gamma\Big(j+1+\frac{r}{2}\Big) \beta(r+1) m^{j+1+r/2},
  \end{align*}
  which does not depend on the parity of $2m$, as the respective
  residue of $\kappa_{2m}$ is $\frac{1}{2}$ in either case. Finally, the
  $O$-term in \eqref{eq:asymptotic-summands-T} comes from the fact
  that $f^{*}$ may be continued analytically arbitrarily far to the
  left in the complex plane without encountering any
  additional poles. 
\end{proof}
\begin{remark}
  In Lemma~\ref{lemma:summands-T}, particular values of the Dirichlet
  beta function are required. To compute the asymptotic expansions for
  the first moments, we need $\beta(1) = \pi/4$, $\beta(2) = G \approx
  0.91597$, as well as $\beta(3) = \pi^{3}/32$, where $G$ is the
  Catalan constant. These values are taken from
  \cite[Table~3.7.1]{Oldham-Spanier:1987:atlas-funct}.  
\end{remark}

At this point, all that remains to obtain asymptotic expansions is to
multiply the contributions resulting from Lemma~\ref{lemma:summands-T}
with the correct coefficients and contributions
from~\eqref{eq:prob-T-asy}. 

\begin{theorem}[Asymptotic analysis of admissible random walks on $\N_{0}$]\label{thm:asy-NN}
  The probability that a random walk on $\N_{0}$ is admissible can be
  expressed asymptotically as
  \begin{equation}\label{eq:asy-p-prob}
    p_{n} = \sqrt{\frac{\pi}{2n}} -
  \frac{5\sqrt{2\pi}}{24\sqrt{n^{3}}} +
  \frac{127\sqrt{2\pi}}{960\sqrt{n^{5}}} - \frac{1571
    \sqrt{2\pi}}{16128\sqrt{n^{7}}} -
  \frac{1896913\sqrt{2\pi}}{184320\sqrt{n^{9}}} +
  O\Big(\frac{1}{\sqrt{n^{11}}}\Big),
  \end{equation}
  where $\sqrt{\pi/2} \approx 1.25331$. The expected height of
  admissible random walks is given by 
  \begin{equation}\label{eq:asy-p-exp}
    \E H_{n} = 2G \sqrt{\frac{2n}{\pi}} - 1 + \frac{5\sqrt{2}
    G}{6\sqrt{\pi n}} - \frac{131 \sqrt{2} G}{720 \sqrt{\pi n^{3}}} +
  \frac{1129\sqrt{2} G}{12096 \sqrt{\pi n^{5}}} +
  O\Big(\frac{1}{\sqrt{n^{7}}}\Big),
  \end{equation}
  where $2G\sqrt{2/\pi} \approx 1.46167$, and the variance of $H_{n}$
  can be expressed as 
  \begin{equation}\label{eq:asy-p-var}
    \V H_{n} = \frac{\pi^{3} - 32G^{2}}{4\pi} n + \frac{\pi^{3} -
    40G^{2}}{6\pi} - \frac{\pi^{3} - 12G^{2}}{180 \pi n} + \frac{11
    \pi^{3} - 265 G^{2}}{1890 \pi n^{2}} +
  O\Big(\frac{1}{n^{3}}\Big),
  \end{equation}
  where $(\pi^{3} - 32G^{2})/(4\pi) \approx 0.33092$. Generally, the $r$-th moment is asymptotically given by
\begin{equation}
\E H_n^r \sim \frac{2^{r/2+2}}{\pi} \Gamma \Big( \frac{r}{2} + 1 \Big) \beta(r+1) n^{r/2}.
\end{equation}
Moreover, if $\eta = h/\sqrt{n}$ satisfies $3/\sqrt{\log n} < \eta < \sqrt{\log n}/2$ and $h \equiv n \bmod 2$, we have the local limit theorem
\begin{align}\label{eq:theta1_p}
\Prob(H_n = h) = \frac{p_n^{(h)}}{p_n} \sim \frac{2\phi(\eta)}{\sqrt{n}} &= \frac{8\eta}{\pi \sqrt{n}} \sum_{k \geq 0} (-1)^k (2k+1) \exp \Big( -\frac{(2k+1)^2\eta^2}{2} \Big) \\\label{eq:theta2_p}
&= \frac{2\sqrt{2\pi}}{\eta^2\sqrt{n}} \sum_{k \geq 0} (-1)^k (2k+1) \exp \Big( -\frac{\pi^2(2k+1)^2}{8\eta^2} \Big).
\end{align}
\end{theorem}

\begin{remark}
The fact that the two series in~\eqref{eq:theta1_p} and~\eqref{eq:theta2_p} that represent the density $\phi(\eta)$ are equal is a simple consequence of the Poisson sum formula. We also note that the asymptotic behavior of the moments of $H_n$ readily implies that the normalized random variable $H_n/\sqrt{n}$ converges weakly to the distribution whose density is given by $\phi(\eta)$ (see~\cite[Theorem C.2]{Flajolet-Sedgewick:ta:analy}). The local limit theorem~\eqref{eq:theta1_p} is somewhat stronger.
\end{remark}

\begin{proof}
  With \eqref{eq:prob-T-asy} and the result of
  Lemma~\ref{lemma:summands-T}, obtaining an asymptotic expansion of
  $p_{2m-1}$ is only a question of developing the shifted central
  binomial coefficient and multiplying with the correct growth
  contributions from \eqref{eq:asymptotic-summands-T}. By doing so
  (with the help of SageMath
  \cite{Stein-others:2015:sage-mathem-6.5}: a worksheet containing
  these computations as well as some numerical comparisons can be found at
  \url{http://arxiv.org/src/1503.08790/anc/random-walk_NN.ipynb}), an
  asymptotic expansion in the half-integer $m$ is
  obtained. Substituting $m = (n+1)/2$ then gives \eqref{eq:asy-p-prob}. 

  The results in \eqref{eq:asy-p-exp} and \eqref{eq:asy-p-var} are
  obtained by considering
  \[ \E(H_{n}+1)^{r} = \frac{\sum_{h\geq 0} (h+1)^{r}
    p_{n}^{(h)}}{p_{n}},  \]
  making use of \eqref{eq:prob-T-asy} and Lemma~\ref{lemma:summands-T} again. Note that we have $\E H_{n} = \E(H_{n}+1) - 1$,
  as well as $\V H_{n} = \E(H_{n}+1)^{2} - [\E(H_{n}+1)]^{2}$. For higher moments, we only give the principal term of the asymptotics, which corresponds to the coefficient $c_{00}$ in~\eqref{eq:prob-T-asy}, but in principle it would be possible to calculate further terms as well.

It remains to prove~\eqref{eq:theta1_p}. To this end, we revisit the explicit expression (recall that we set $n = 2m-1$)
\begin{equation*}
    p_{2m-1}^{(h)} = \frac{4}{4^{m}} \sum_{k\geq 0} (-1)^{k}
    \frac{\tau_{h,k}}{m} \binom{2m}{m-\tau_{h,k}}.
\end{equation*}
First of all, we can eliminate all $k$ with $\tau_{h,k} > m^{2/3}$, since their total contribution is at most $O(m \exp(-m^{1/3}))$ as before. For all other values of $k$, we replace the binomial coefficient according to Lemma~\ref{lemma:central-binom-asy} by
$$\binom{2m}{m-\tau_{h,k}} = \frac{4^m}{\sqrt{\pi m}} \exp  \Big( -\frac{\tau_{h,k}^2}{m} \Big) \Big(1 + O \Big( \frac{1+\tau_{h,k}}{m} \Big) \Big).$$
Note here that
$$\tau_{h,k} = \frac{(h+1)(2k+1)}2 = \frac{h}{2} (2k+1) \Big( 1 + O\Big( \frac{1}{h} \Big) \Big),$$
and likewise
$$\frac{\tau_{h,k}^2}{m} = \frac{h^2(2k+1)^2}{2n}  \Big( 1 + O \Big( \frac{1}{h} + \frac{1}{n}\Big) \Big).$$
It follows that
$$ \frac{\tau_{h,k}}{m}  \exp  \Big( - \frac{\tau_{h,k}^2}{m} \Big) = \frac{h(2k+1)}{n} \exp \Big( - \frac{h^2(2k+1)^2}{2n} \Big) \Big( 1 + O \Big( \frac{1}{h} + \frac{hk^2+1}{n}\Big) \Big).$$
We are assuming that $\tau_{h,k} \leq m^{2/3} = ((n+1)/2)^{2/3}$, which implies $hk^2/n = O(n^{1/3}/h)$. In view of our assumptions on $h$, this means that the error term is $O(n^{-1/6} \sqrt{\log n})$. Thus we have
\begin{multline*}
p_n^{(h)} = p_{2m-1}^{(h)} = \frac{4\sqrt{2}h}{\sqrt{\pi n^3}}\\\times \sum_{\substack{k \geq 0 \\ \tau_{h,k} \leq ((n+1)/2)^{2/3}}} (-1)^k (2k+1)  \exp \Big( - \frac{h^2(2k+1)^2}{2n} \Big) \Big( 1 + O \Big( \frac{\sqrt{\log n}}{n^{1/6}}\Big)\Big) \\
+ O\big(n \exp(-(n/2)^{1/3}) \big).
\end{multline*}
Adding all terms $\tau_{h,k} > m^{2/3} = ((n+1)/2)^{2/3}$ back only results in a negligible contribution that decays faster than any power of $n$ again, but we need to be careful with the $O$-term inside the sum, as we have to bound the accumulated error by the sum of the absolute values. We have
$$\sum_{k \geq 0} (2k+1) \exp \Big( - \frac{h^2(2k+1)^2}{2n} \Big) = O(n/h^2),$$
which can be seen e.g. by approximating the sum by an integral (or by means of the Mellin transform again), so 
\begin{align*}
p_n^{(h)} &= \frac{4\sqrt{2}h}{\sqrt{\pi n^3}} \sum_{k \geq 0} (-1)^k (2k+1)  \exp \Big(- \frac{h^2(2k+1)^2}{2n} \Big) + O\Big( \frac{\sqrt{\log n}}{hn^{2/3}} \Big) \\
&=\frac{4\sqrt{2}\eta}{\sqrt{\pi} n} \sum_{k \geq 0} (-1)^k (2k+1)  \exp \Big(- \frac{h^2(2k+1)^2}{2n} \Big) + O\Big( \frac{\log n}{n^{7/6}} \Big).
\end{align*}
Since $p_n = \sqrt{\frac{\pi}{2n}}(1+O(n^{-1}))$, this yields
\begin{align*}
\frac{p_n^{(h)}}{p_n} &= \frac{8\eta}{\pi \sqrt{n}} \sum_{k \geq 0} (-1)^k (2k+1)  \exp \Big(- \frac{\eta^2(2k+1)^2}{2} \Big) + O\Big( \frac{\log n}{n^{2/3}} \Big) \\
&= \frac{2\phi(\eta)}{\sqrt{n}} + O\Big( \frac{\log n}{n^{2/3}} \Big).
\end{align*}
For $\eta \geq 1$, the sum is bounded below by a constant multiple of $\exp(-\eta^2/2)$ (as can be seen by bounding the sum of all terms with $k \geq 1$), which in turn is at least $\exp(-(\log n)/8) = n^{-1/8}$ by our assumptions on $\eta$. Thus the first term indeed dominates the error term in this case. If $\eta < 1$, we use the alternative representation~\eqref{eq:theta2_p}, which shows that $\phi(\eta)$ is bounded below by a constant multiple of $\eta^{-2} \exp(-\pi^2/(8\eta)^2)$. This in turn is at least $(1/9)n^{-\pi^2/72}\log n$ by the assumptions on $\eta$, and since $\pi^2/72 < 1/6$, we can draw the same conclusion.
\end{proof}
\begin{remark}
  As stated in the introduction, the number $2^{n}p_{n}$ gives the
  number of extremal lattice paths on $\Z$---and thus, with the
  asymptotic expansion of $p_{n}$, we also have an asymptotic
  expansion for the number of extremal lattice paths on $\Z$ of given
  length.
\end{remark}

This concludes our analysis of admissible random walks on
$\N_{0}$. In the next section, we investigate admissible random walks
on $\Z$.

\section{Ballot Sequences and Admissible Random Walks on
  \texorpdfstring{$\Z$}{Z}}\label{sec:RW-Z} 

In principle, the approach we follow for the analysis of the
asymptotic behavior of admissible random walks on $\Z$ is the same as
in the previous section. However, due to the different structure of
\eqref{eq:prob-expl-U}, some steps will need to be adapted.

With the notation of Lemma~\ref{lemma:central-binom-asy}, we are able
to express $q_{2m-2}$ for a half-integer $m\in \frac{1}{2}\N$ with
$m\geq 1$ as  
\begin{equation}\label{eq:prob-U-asy}
q_{2m-2} \sim \frac{4}{\sqrt{m\pi}} \frac{1}{2m-1}\hspace{-1em} \sum_{\substack{h,k\geq 0
  \\ h\equiv 2m\bmod 2}} \hspace{-1em} \frac{2\upsilon_{h,k}^{2} - m}{m}
\exp\Big(- \frac{\upsilon_{h,k}^{2}}{m}\Big) \sum_{\ell,
  j\geq 0} c_{\ell, j} \frac{\upsilon_{h,k}^{2j}}{m^{\ell}}.
\end{equation}
In analogy to our investigation of admissible random
walks on $\N_{0}$, we also want to determine the expected height and
variance of admissible random walks. These are related to the random
variable $\widetilde H_{n}$, which we defined by
\[ \Prob(\widetilde H_{n} = h) = \frac{q_{n}^{(h)}}{q_{n}}.  \]
To make things easier, we will investigate moments of the form
$\E(\widetilde H_{n} + 2)^{r}$. They can be computed by
\[ \E(\widetilde H_{n} + 2)^{r} = \sum_{h\geq 0} (h+2)^{r} \Prob(\widetilde
H_{n} = h) = \frac{\sum_{h\geq 0} (h+2)^{r} q_{n}^{(h)}}{q_{n}}.  \]
Therefore, we are interested in the asymptotic contribution of
\[ \sum_{\substack{h,k\geq 0 \\ h\equiv 2m\bmod 2}} \hspace{-1em}
\frac{2\upsilon_{h,k}^{2} - m}{m} \upsilon_{h,k}^{2j} (h+2)^{r}
\exp\Big(-\frac{\upsilon_{h,k}^{2}}{m}\Big),  \] 
which is discussed in the following lemma.

\begin{lemma}
  Let $K > 0$ be fixed. Then we have the asymptotic expansion
  \begin{equation}\label{eq:q-trans-1}
    \sum_{\substack{h,k\geq 0 \\ h\equiv 2m\bmod 2}} \hspace{-1em}
    \frac{2\upsilon_{h,k}^{2} - m}{m}
    \exp\Big(-\frac{\upsilon_{h,k}^{2}}{m}\Big) = \frac{\sqrt{m\pi}}{4}  + O(m^{-K}).
  \end{equation}
  For $j\in \N$ we have
  \begin{multline}\label{eq:q-trans-2}
    \sum_{\substack{h,k\geq 0 \\ h\equiv 2m\bmod 2}} \hspace{-1em}
    \frac{2\upsilon_{h,k}^{2} - m}{m} \upsilon_{h,k}^{2j}
    \exp\Big(-\frac{\upsilon_{h,k}^{2}}{m}\Big) \\ =  
    \Big(\frac{\log m}{2} + 2\gamma + \log 2 + \frac{1}{2} \psi\Big(j +
    \frac{1}{2}\Big) + \frac{1}{2j} + \llbracket m\not\in \N \rrbracket\cdot (2\log
    2 - 2)\Big) \\ \times \frac{j}{2} \Gamma\Big(j + \frac{1}{2}\Big)
    m^{j+1/2} + O(m^{-K})
  \end{multline}
  where $\psi(s)$ is the digamma function. Finally, for
  $j\in\N_{0}$, $r\in \N$ we find 
  \begin{multline}\label{eq:q-trans-3}
    \sum_{\substack{h,k\geq 0 \\ h\equiv 2m\bmod 2}} \hspace{-1em}
    \frac{2\upsilon_{h,k}^{2} - m}{m} \upsilon_{h,k}^{2j} (h+2)^{r}
    \exp\Big(-\frac{\upsilon_{h,k}^{2}}{m}\Big) \\
    = j \Gamma\Big(j + \frac{1}{2}\Big) \kappa_{2m}(1-r)
    m^{j+1/2} \\ + \frac{1}{2} \Big(j + \frac{r}{2}\Big)
    \Gamma\Big(j + \frac{r+1}{2}\Big) (2^{r+1}
    - 1) \zeta(r+1) m^{j+(r+1)/2} + O(m^{-K}),
  \end{multline}
  where $\kappa_{2m}(s) = 2^{-s} \zeta(s)$ for $m\in \N$ and
  $\kappa_{2m}(s) = (1-2^{-s}) \zeta(s) - 1$ otherwise.
\end{lemma}
\begin{proof}
  Let $j$, $r\in \N_{0}$. We want to analyze the sum
  \[ \sum_{\substack{h,k\geq 0 \\ h\equiv 2m\bmod 2}} \hspace{-1em}
  \frac{2\upsilon_{h,k}^{2} - m}{m} \upsilon_{h,k}^{2j} (h+2)^{r}
  \exp\Big(-\frac{\upsilon_{h,k}^{2}}{m}\Big) \]
  asymptotically, where $m$ is a half-integer in
  $\frac{1}{2}\N$ with $m\geq 1$. 

In analogy to the proof of Lemma~\ref{lemma:summands-T}, we substitute $x^{-2} = m$, so that the sum becomes
  \begin{align*}
    f(x) & := \sum_{\substack{h,k\geq 0 \\ h\equiv 2m \bmod 2}}
           \hspace{-1em} (2x^{2} \upsilon_{h,k}^{2} - 1) \upsilon_{h,k}^{2j}
           (h+2)^{r} \exp(-\upsilon_{h,k}^{2} x^{2}) \\
         & = 2x^{2} \hspace{-1em}\sum_{\substack{h,k\geq 0 \\ h\equiv 2m \bmod 2}}
           \hspace{-1em}  \upsilon_{h,k}^{2j+2}
           (h+2)^{r} \exp(-\upsilon_{h,k}^{2} x^{2})  -\hspace{-1em}
           \sum_{\substack{h,k\geq 0 \\ h\equiv 2m \bmod 2}}
           \hspace{-1em} \upsilon_{h,k}^{2j} (h+2)^{r}
           \exp(-\upsilon_{h,k}^{2} x^{2}) \\
         & =: 2x^{2} f_{1}(x) - f_{2}(x).
  \end{align*}
Both $f_1$ and $f_2$ are harmonic sums, and we determine their Mellin transforms as we did earlier in the proof of Lemma~\ref{lemma:summands-T}. 
  By elementary properties of the Mellin transform, we know that
  $f^{*}(s) = 2f_{1}^{*}(s+2) - f_{2}^{*}(s)$. Let $\Lambda_{1}$ and
  $\Lambda_{2}$ be the Dirichlet series associated with the harmonic sums
  $f_{1}(x)$ and $f_{2}(x)$, respectively. We find
  \begin{align*}
    \Lambda_{1}(s) & = \hspace{-1em}\sum_{\substack{h,k\geq 0 \\ h\equiv 2m \bmod 2}}
           \hspace{-1em} \upsilon_{h,k}^{2j+2-s} (h+2)^{r} = 
                     2^{s-(2j+2)} \hspace{-1em} \sum_{\substack{h,k\geq 0 \\
                     h\equiv 2m \bmod 2}} \hspace{-1em}
                     (h+2)^{2j+2+r-s} (2k+1)^{2j+2-s}\\
                   & = (2^{s-(2j+2)} - 1) 
                     \zeta(s-(2j+2)) \hspace{-1em}\sum_{\substack{h\geq 0
                     \\ h\equiv 2m \bmod 2}} \hspace{-1em}
                     (h+2)^{2j+2+r-s}. 
  \end{align*}
  We investigate the sum over $h$ separately, and obtain
  \[ \kappa_{2m}(s) := \sum_{\substack{h\geq 0 \\ h\equiv 2m\bmod
    2}}\hspace{-1em} (h+2)^{-s} = \begin{cases} 2^{-s} 
    \zeta(s) & \text{ for } m\in \N, \\
    (1-2^{-s})\zeta(s) - 1 & \text{ for } m\not\in \N. \end{cases} \]
  Therefore, we find the Mellin transform of the first harmonic sum to be
  \begin{align*} 
    f_{1}^{*}(s) &= \frac{1}{2}\Gamma\Big(\frac{s}{2}\Big)
                   \Lambda_{1}(s) = \frac{1}{2} \Gamma\Big(\frac{s}{2}\Big)
                   (2^{s-2j-2} -1) \zeta(s-2j-2) \kappa_{2m}(s-(2j+r)-2).
  \end{align*}
  The Mellin transform of the second sum can be found in a completely
  analogous way: we have
  \[ f_{2}^{*}(s) = \frac{1}{2} \Gamma\Big(\frac{s}{2}\Big)
  (2^{s-2j} - 1) \zeta(s-2j) \kappa_{2m}(s - (2j+r)). \]
  Altogether, this yields the Mellin transform
  \begin{align*}
    f^{*}(s) & = 2 f_{1}^{*}(s+2) - f_{2}^{*}(s) \\ 
             & = \frac{s-1}{2} \Gamma\Big(\frac{s}{2}\Big)
               (2^{s-2j} -1) \zeta(s-2j) \kappa_{2m}(s-(2j+r)).
  \end{align*}
  As in the proof of Lemma~\ref{lemma:summands-T}, the growth
  conditions necessary for application of the converse mapping theorem
  \cite[Theorem~4]{Flajolet-Gourdon-Dumas:1995:mellin} hold. 

  In order to analyze the poles of $f^{*}(s)$, we need to distinguish
  three cases, as $\zeta(s-2j)$ has a simple pole at $s=2j+1$ and
  $\kappa_{2m}(s-(2j+r))$ has a simple pole at $s=2j+r+1$. The poles
  of $\Gamma(s/2)$ at even $s \leq 0$ are canceled by the zeros of
  $\zeta(s-2j)$, unless $s = j = 0$. In that case, the pole is
  canceled by the factor $(2^{s-2j} - 1)$.

  First, let $r = j = 0$. Then, $f^{*}(s)$ has a simple pole at $s =
  1$, because one of the poles of $\zeta(s)$ or $\kappa_{2m}(s)$
  cancels against the zero of $(s-1)$. Here, the residue of $f^{*}(s)$
  is given by $\sqrt{\pi}/4$, which translates to a contribution of
  $\sqrt{m \pi}/4$. This proves \eqref{eq:q-trans-1}.

  Second, for $r = 0$ and $j > 0$, the function $f^{*}(s)$ has a
  pole of order $2$ at $s = 2j+1$. By expanding all the occurring
  functions, we find the Laurent expansion
  \[ f^{*}(s) \asymp \begin{cases} 
    \frac{j}{2} \Gamma\left(j + \frac{1}{2}\right)
    \Big[\frac{1}{(s-(2j+1))^{2}} + \frac{\frac{1}{2} \psi(j+\frac{1}{2}) +
        2\gamma + \log 2 + \frac{1}{2j}}{s - (2j+1)}\Big] + O(1) 
    & \text{ for } m\in \N, \\
    \frac{j}{2} \Gamma\left(j + \frac{1}{2}\right)
    \Big[\frac{1}{(s-(2j+1))^{2}} + \frac{\frac{1}{2} \psi(j+\frac{1}{2}) +
        2\gamma + 3\log 2 - 2 + \frac{1}{2j}}{s - (2j+1)}\Big] + O(1)
    & \text{ for } m\not\in \N,
  \end{cases} \]
  where $\psi(s)$ is the digamma function (cf.~\DLMF{5.2}{2}, see
  \cite[\href{http://dlmf.nist.gov/5.4.ii}{\S 5.4(ii)}]{NIST:DLMF} for
  special values). As the pole of order $2$ contributes the factor
  $\frac{1}{2} m^{j+1/2} \log m$, and the pole of order $1$ gives
  $m^{j+1/2}$, \eqref{eq:q-trans-2} is proved. 

  Finally, consider $r > 0$. In this case we have two separate single
  poles at $s = 2j+1$ and $s = 2j+r+1$. Computing the residues gives
  the growth contribution
  \begin{multline*} j\Gamma\Big(j + \frac{1}{2}\Big) \kappa_{2m}(1-r)
  m^{j+1/2}  + \Big(j+ \frac{r}{2}\Big)\Gamma\Big(j +
    \frac{r+1}{2}\Big) \Big(2^{r} - \frac{1}{2}\Big) \zeta(r+1)
  m^{j+(r+1)/2},
  \end{multline*}
  which proves~\eqref{eq:q-trans-3}.
\end{proof}

Fortunately, when explicitly computing the expansion, all the logarithmic terms cancel out
and we obtain the same behavior for admissible paths of even and odd length. The following
theorem summarizes our findings.

\begin{theorem}[Asymptotic analysis of admissible random walks on
  $\Z$]\label{thm:asy-ZZ}
  The probability that a random walk on $\Z$ is admissible has the
  asymptotic expansion
  \begin{equation}\label{eq:asy-q-prob}
    q_{n} = \frac{1}{n} - \frac{4}{3n^{2}} + \frac{88}{45n^{3}} -
    \frac{976}{315n^{4}} +  \frac{3488}{675n^{5}} - \frac{276928}{31185n^{6}} + O\Big(\frac{1}{n^{7}}\Big).
  \end{equation}
  The expected height of admissible random walks on $\Z$ is given by
  \begin{equation}
    \E\widetilde H_{n}  = \frac{\sqrt{2 \pi^{3}}}{4} \sqrt{n} - 2 +
  \frac{3 \sqrt{2 \pi^{3}}}{16 \sqrt{n}} - \frac{539\sqrt{2
      \pi^{3}}}{5760 \sqrt{n^{3}}} +
  \frac{50713\sqrt{2\pi^{3}}}{483840 \sqrt{n^{5}}} +
  O\Big(\frac{1}{\sqrt{n^{7}}}\Big),
  \end{equation}
  where $\sqrt{2\pi^{3}}/4 \approx 1.96870$, and the variance of
  $\widetilde H_{n}$ can be expressed as 
  \begin{multline} 
    \V \widetilde H_{n} = \frac{28\zeta(3) - \pi^{3}}{8} n +
    \frac{224\zeta(3) - 9\pi^{3}}{48} - \frac{1792\zeta(3) -
      67\pi^{3}}{2880n} \\ + \frac{107520\zeta(3) - 4189\pi^{3}}{120960
      n^{2}} + O\Big(\frac{1}{n^{3}}\Big),
  \end{multline}
  where $(28\zeta(3) - \pi^{3})/8 \approx 0.33141$. Generally, the $r$-th moment is asymptotically given by
\begin{equation}
\E \widetilde H_n^r \sim \frac{r}{\sqrt{\pi}} \Gamma \Big( \frac{r+1}{2} \Big) (2^{r+1}-1)2^{-r/2} \zeta(r+1) n^{r/2}.
\end{equation}
Moreover, if $\eta = h/\sqrt{n}$ satisfies $6/\sqrt{\log n} < \eta < \sqrt{\log n}/2$, we have the local limit theorem
\begin{align}\label{eq:theta1}
\Prob(\widetilde{H}_n = h) = \frac{q_n^{(h)}}{q_n} \sim \frac{2\chi(\eta)}{\sqrt{n}} &= \frac{4\sqrt{2}}{\sqrt{\pi n}} \sum_{k \geq 0} ((2k+1)^2 \eta^2-1) \exp \Big( -\frac{(2k+1)^2\eta^2}{2} \Big) \\\label{eq:theta2}
&= \frac{4\pi^2}{\eta^3\sqrt{n}} \sum_{k \geq 1} (-1)^{k-1} k^2 \exp \Big( -\frac{\pi^2k^2}{2\eta^2} \Big).
\end{align}
\end{theorem}

\begin{remark}
Again, the two expressions for the limiting density $\chi$ are equivalent, as can be seen by an application of the Poisson sum formula.
\end{remark}

\begin{proof}
  Analogous to Theorem~\ref{thm:asy-NN}. The asymptotic expansions
  were again computed with the help of SageMath
  \cite{Stein-others:2015:sage-mathem-6.5}, and a corresponding
  worksheet (containing these computations as well as some numerical comparisons) can be
  found at \url{http://arxiv.org/src/1503.08790/anc/random-walk_ZZ.ipynb}.
\end{proof}
\begin{remark}
  As every simple symmetric random walk of length $n$ on $\Z$ occurs
  with probability $2^{-n}$, we know that the number of admissible
  random walks on $\Z$ is $2^{n} q_{n}$. Thus, an asymptotic expansion
  for the number of admissible random walks follows directly
  from~\eqref{eq:asy-q-prob} upon multiplication by $2^{n}$. This
  is sequence \href{http://oeis.org/A167510}{A167510} in \cite{OEIS:2015}.
\end{remark}

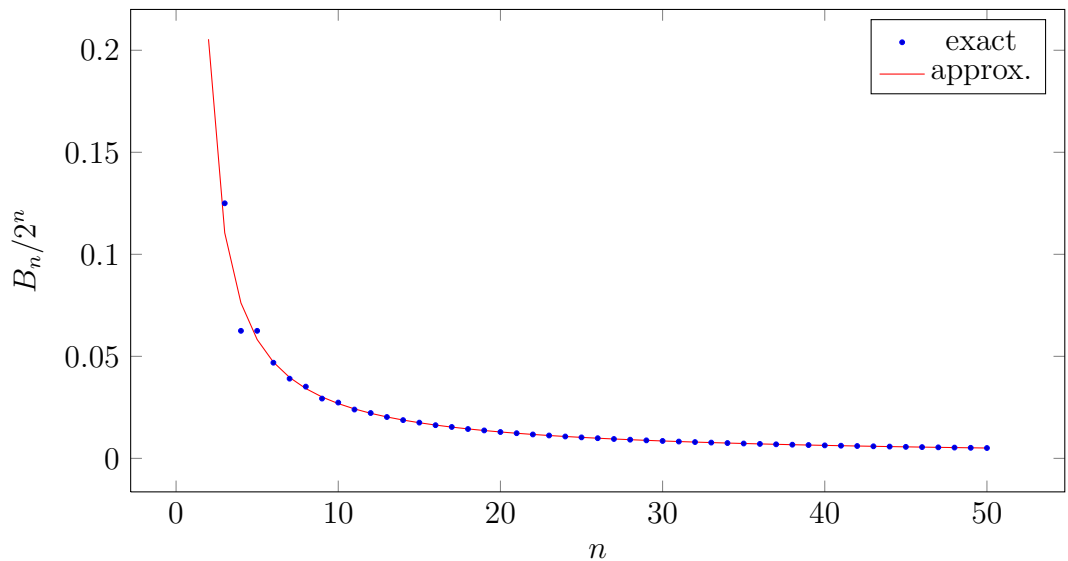
\begin{figure}[b]
  \centering
  \begin{tikzpicture}
    \begin{axis}[xlabel={$n$}, ylabel={$B_{n}/2^{n}$}, width=14cm, height=8cm, ymax=0.22,
      yticklabel style={/pgf/number format/fixed, /pgf/number format/precision=3},
      xtick = {0,10,...,50}, legend entries = {exact, approx.}
      ]
      \addplot+[only marks, mark size=0.9pt] table[x=n, y=exact] {ballot_cmp.dat};
      \addplot+[no marks] table[x=n, y=asy] {ballot_cmp.dat};
    \end{axis}
  \end{tikzpicture}
  \caption{Numerical approximation of $B_n/2^n$}
  \label{fig:ballot-numerical}
\end{figure}

Furthermore, in the introduction we illustrated that admissible random
walks are strongly related to bidirectional ballot sequences. Since
every bidirectional ballot sequence of length $n+2$ 
corresponds to an admissible random walk of length $n$ on $\Z$ (i.e.,
$B_{n} = 2^{n-2} q_{n-2}$), we are able to prove Zhao's conjecture that was mentioned in
the introduction.

\begin{corollary}[Bidirectional ballot walks]\label{cor:zhao}
  The number of bidirectional ballot walks $B_{n}$ of length $n$ can
  be expressed asymptotically as
  \begin{equation}
    \label{eq:ballot-asy}
    B_{n} = 2^{n} \Big(\frac{1}{4n} + \frac{1}{6 n^{2}} +
    \frac{7}{45 n^{3}} + \frac{10}{63 n^{4}} + \frac{764}{4725n^{5}} +
    \frac{4952}{31185n^{6}}\Big) + O\Big(\frac{2^{n}}{n^{7}}\Big).
  \end{equation}
\end{corollary}

In Figure~\ref{fig:ballot-numerical} we compare the exact values of $B_{n}/2^{n}$ with the
values obtained from the asymptotic expansion in~\eqref{eq:ballot-asy}.

\bibliographystyle{amsplainurl}
\bibliography{bib/cheub}

\end{document}